\documentclass[12pt]{amsart}
\usepackage{amsthm, amssymb}
\usepackage{amsfonts, amscd}
\usepackage{epsfig,multicol}
\theoremstyle{plain} \numberwithin{equation}{section}
\newtheorem{theo}{Theorem}[section]
\newtheorem{coro}[theo]{Corollary}
\newtheorem{prop}[theo]{Proposition}
\newtheorem{lemm}[theo]{Lemma}

\theoremstyle{definition}
\newtheorem*{defi}{Definition}
\newtheorem*{exam}{Example}
\newtheorem*{rema}{Remark}

\def\Z{\mathbb Z}
\def\C{\mathbb C}
\def\R{\mathbb R}

\def\G{\Gamma}

\def\G0{G^0}
\def\NGT{N_G(T)}
\def\NG0T{N_{\G0}(T)}

\def\g{\mathfrak g}
\def\V{\mathcal{V}}
\def\VC{\V_{\C}}
\def\RG{\Delta(G)}
\def\rhog{{\rho_g}}
\def\rhoP{\mathcal D}
\def\RM{R(P)}

\DeclareMathOperator{\Aut}{Aut}

\DeclareMathOperator{\Ker}{Ker}

\DeclareMathOperator{\U}{U}

\DeclareMathOperator{\Hom}{Hom}

\DeclareMathOperator{\Fix}{Fix}
\DeclareMathOperator{\Sympl}{Symp}

\begin{document}
\title{Symmetry of a symplectic toric manifold}
\author[M. Masuda]{Mikiya Masuda}
\address{Department of Mathematics, Osaka City University, Sumiyoshi-ku, Osaka 558-8585, Japan.}
\email{masuda@sci.osaka-cu.ac.jp}

\date{\today}
\thanks{The author was partially supported by Grant-in-Aid for Scientific Research 19204007}

\begin{abstract}
The action of a torus group $T$ on a symplectic toric manifold $(M,\omega)$ often extends to an effective action of a 
(non-abelian) compact Lie group $G$.  We may think of $T$ and $G$ as compact Lie subgroups of 
the symplectomorphism group $\Sympl(M,\omega)$ of $(M,\omega)$.  
On the other hand, $(M,\omega)$ is determined by the associated moment polytope $P$ by the result of Delzant \cite{delz88}.  
Therefore, the group $G$ should be estimated in terms of $P$ or we may say that a \emph{maximal} compact Lie subgroup of 
$\Sympl(M,\omega)$ containing the torus $T$ should be described in terms of $P$. 

In this paper, we introduce a root system $R(P)$ associated to $P$ 
and prove that any irreducible subsystem of $R(P)$ is of type A and 
the root system $\Delta(G)$ of the group $G$ is a subsystem of $R(P)$ 
(so that $R(P)$ gives an upper bound for the identity component of $G$ 
and any irreducible factor of $\Delta(G)$ is of type A). 
We also introduce a homomorphism $\rhoP$ from the normalizer $\NGT$ of $T$ in $G$ to 
an automorphism group $\Aut(P)$ of $P$, which detects the connected components of $G$.  
Finally we find a maximal compact Lie subgroup $G_{\max}$ of $\Sympl(M,\omega)$ containing 
the torus $T$. 
\end{abstract}

\maketitle 

\section{Introduction}

A \emph{symplectic toric manifold} is a compact connected symplectic manifold $(M,\omega)$ with an effective 
Hamiltonian action of a torus group $T$ of half the dimension of the manifold $M$. 
Delzant \cite{delz88} proves that $M$ is equivariantly diffeomorphic to a smooth projective toric variety with 
the restricted $T$-action.  
Moreover he classifies symplectic toric manifolds by showing that the correspondence 
from symplectic toric manifolds to their moment polytopes 
is one-to-one.  Therefore, all geometrical information on $(M,\omega)$ is 
encoded in the moment polytope $P$ associated with $(M,\omega)$.  

The $T$-action on $(M,\omega)$ often extends to an effective action of a (non-abelian) compact 
Lie group $G$.  We may think of $T$ and $G$ as compact Lie subgroups of the symplectomorphism 
group $\Sympl(M,\omega)$ of $(M,\omega)$.  Since the $T$-fixed point set in $M$ is non-empty and $\dim T=\frac{1}{2}\dim M$, 
there is no torus subgroup of $\Sympl(M,\omega)$ containing $T$ properly.  
This means that $T$ is a maximal torus of $G$. 

In this paper, we introduce a root system $\RM$ associated to the moment polytope $P$ 
and a homomorphism 
\begin{equation} \label{m1}
\rhoP\colon \NGT\to \Aut(P), 
\end{equation}
where $\NGT$ denotes the normalizer of $T$ in $G$ 
and $\Aut(P)$ denotes an automorphism group of $P$. 
It turns out that the root system $\RM$ gives information on the identity component $\G0$ of 
$G$ and the homomorphism $\rhoP$ induces an injective homomorphism
\[
G/\G0\cong \NGT/\NG0T\to \Aut(P)/\rhoP(\NG0T)
\]
so that $\rhoP$ detects the connected components of $G$.    
Here is a summary of our results. 

\begin{theo} \label{main}
Let $(M,\omega)$ be a symplectic toric manifold with a Hamiltonian action of a torus $T$ 
of half the dimension of the dimension of $M$ and let $P$ be the associated moment polytope. 
Then the following hold. 
\begin{enumerate}
\item Any irreducible subsystem of $\RM$ is of type A.  
\item If $G$ is a compact Lie subgroup of $\Sympl(M,\omega)$ containing the torus $T$, 
then the root system $\Delta(G)$ of $G$ is a subsystem of $\RM$,  
so that any irreducible factor of $\Delta(G)$ is of type A by (1) above. 
\item Let $G$ be as in (2) above.  If $\Delta(G)=\RM$ and the homomorphism $\rhoP$ in \eqref{m1} is surjective, then 
$G$ is maximal among compact Lie subgroups of $\Sympl(M,\omega)$, i.e, 
$G$ is not properly contained in a compact Lie subgroup of $\Sympl(M,\omega)$. 
\item There exists a compact Lie subgroup $G_{\max}$ of $\Sympl(M,\omega)$ containing the torus $T$ 
such that the assumption in (3) above is satisfied. 
\end{enumerate}
\end{theo}

\begin{rema}
By (3) above, the group $G_{\max}$ in (4) above, which contains the torus $T$, 
is maximal among compact Lie subgroups of $\Sympl(M,\omega)$.  
However, the author does not know whether any maximal compact 
Lie subgroup of $\Sympl(M,\omega)$ containing $T$ is conjugate to $G_{\max}$ in $\Sympl(M,\omega)$, 
where the torus $T$ is fixed.  Related to this question, it is proved in \cite{ka-ke-pi06} that when $\dim M=4$, 
the number of conjugacy classes of 2-dimensional tori in $\Sympl(M,\omega)$ is finite. 
\end{rema} 

Our work is motivated by the work of Demazure \cite{dema70} (see also \cite{cox95} or \cite[Section 3.5]{oda88}). 
He introduces a root system $R(\Delta)$ for a complete non-singular fan $\Delta$ and proves that 
it agrees with the root system of the automorphism group $\Aut(X(\Delta))$ of the compact smooth toric variety 
$X(\Delta)$ associated with $\Delta$, where $\Aut(X(\Delta))$ is known to be an algebraic group, and that 
$R_s(\Delta):=R(\Delta) \cap (-R(\Delta))$ is the root system of the reductive (or semisimple) part of $\Aut(X(\Delta))$.  
The symplectic toric manifold $(M,\omega)$ is equivariantly diffeomorphic to a 
smooth projective toric variety $X$ with the restricted $T$-action as mentioned before 
and the fan $\Delta_X$ of the $X$ is the so-called normal fan derived from the moment polytope $P$ associated with $(M,\omega)$, 
where normal vectors $v_i$'s to facets of $P$ are edge vectors in the fan $\Delta_X$.  One sees that our root system $R(P)$ agrees  
with $R_s(\Delta_X)$.  Demazure also describes the connected components of $\Aut(X(\Delta))$ in terms of 
the automorphism group $\Aut(\Delta)$ of the fan $\Delta$.  The automorphism group $\Aut(P)$ of $P$ 
can be regarded as a counterpart to 
$\Aut(\Delta)$, in fact, $\Aut(P)$ can be regarded as a subgroup of $\Aut(\Delta)$.  
We remark that the root systems $R(P)$ and $R(\Delta)$ depend only on the vectors $v_i$'s 
but $\Aut(P)$ and $\Aut(\Delta)$ are not determined by the vectors.  
  
This paper is organized as follows.  In Section~\ref{sect:1} we review an explicit construction 
(called the \emph{Delzant construction} in \cite{guil94}) of a symplectic toric manifold $(M,\omega)$ 
with moment polytope $P$.  
In Section~\ref{sect:2} we rewrite the construction in terms of equivariant (co)homology and also 
recall some facts on the equivariant cohomology of $M$.   
In Section~\ref{sect:3} we make some observations on roots of a compact Lie subgroup $G$ of $\Sympl(M,\omega)$ 
containing the torus $T$.  Based on the observations, we introduce the root system $R(P)$ and prove 
the assertions (1) and (2) in Theorem~\ref{main} (see Proposition~\ref{root}, Theorem~\ref{typeA} and Corollary~\ref{GtypeA}).  
In Section~\ref{sect:5} we find a \emph{connected} compact Lie subgroup $G$ of $\Sympl(M,\omega)$ 
which attains the equality $\Delta(G)=R(P)$.   
In Section~\ref{sect:6} we introduce the homomorphism $\rhoP$ in \eqref{m1} 
for an arbitrary subgroup $G$ of $\Sympl(M,\omega)$ containing $T$ and 
observe that $\rhoP$ detects the connected components of $G$ when $G$ is a compact Lie group.  
The assertions (3) and (4) in Theorem~\ref{main} 
are proved in Section~\ref{sect:7} (see Theorem~\ref{Gmax}).  

Throughout this paper, $(M,\omega)$ will denote a symplectic toric manifold with moment polytope $P$, where 
a Hamiltonian $T$-action on $(M,\omega)$ is incorporated although it is often not mentioned explicitly.  
The argument developed in this paper works for \emph{torus manifolds} introduced in \cite{ha-ma03} 
with some modification.  We will discuss this in a forthcoming paper.

\section{Delzant construction} \label{sect:1}

By the result of Delzant mentioned in the Introduction, a symplectic toric manifold $(M,\omega)$ is 
determined by the associated moment polytope $P$ and is explicitly constructed from $P$.  
We will review the construction in this section. The details can be found in \cite{guil94}. 

Let $\mu\colon M\to \mathfrak t^*$ be a moment map associated with $(M,\omega)$ so that $\mu(M)=P$, 
where $\mathfrak t^*$ is the dual of the Lie algebra $\mathfrak t$ of $T$. 
The moment map $\mu$ is uniquely determined by $(M,\omega)$ up to parallel translations in $\mathfrak t^*$. 
We identify $\mathfrak t$ with $\R^n$ and 
$\mathfrak t^*$ with $(\R^n)^*$ and express  
\begin{equation} \label{P}
P=\{ u\in (\R^n)^*\mid \langle u,v_i\rangle\ge a_i\quad(i=1,\dots,m)\}
\end{equation}
where $v_i\in\Z^n$ is primitive, $\langle\ ,\ \rangle$ is a natural pairing (i.e. evaluation) and $a_i\in \R$.  
Without loss of generality we may assume that there is no redundant inequality in \eqref{P} 
so that the intersection of $P$ with the hyperplane defined by $\langle u,v_i\rangle=a_i$ is a facet (i.e. codimension 1 face) of $P$ 
for each $i$, which we denote by $P_i$.  So there are exactly $m$ facets in $P$.  
The moment polytope $P$ is \emph{non-singular}, which means that $P$ is simple and whenever $n$ facets of $P$ meet at a vertex, 
the $n$ vectors $v_i$'s normal to the $n$ facets form a basis of $\Z^n$.  
A non-singular polytope is called a \emph{Delzant polytope} in \cite{guil94}. 

Let $e_1,\dots,e_m$ be the standard basis of $\Z^m$ and consider the linear map 
$\pi_*\colon \Z^m\to \Z^n$ sending $e_i$ to $v_i$ for $i=1,\dots,m$.  Since $P$ is 
non-singular, $\pi_*$ is surjective and we have an exact sequence
\begin{equation} \label{exactV}
0\to \Ker \pi_*\stackrel{\iota_*}\longrightarrow \Z^m\stackrel{\pi_*}\longrightarrow\Z^n\to 0
\end{equation}
where $\iota_*$ is the inclusion.  Taking the dual of this sequence, we obtain an exact sequence  
\begin{equation} \label{exactV*}
0\leftarrow (\Ker \pi_*)^*\stackrel{\iota^*}\longleftarrow (\Z^m)^*\stackrel{\pi^*}\longleftarrow
(\Z^n)^*\leftarrow 0
\end{equation}
and one easily sees that 
\begin{equation} \label{V}
\pi^*(u)=\sum_{i=1}^m\langle u,v_i\rangle e_i^* \quad\text{for any $u\in (\Z^n)^*$} 
\end{equation}
where $e_1^*,\dots,e_m^*$ denote the dual basis of $e_1,\dots,e_m$. 

The map $\pi_*$ (resp. $\pi^*$) extends to a linear map from $\R^m$ onto $\R^n$ 
 (resp. from $(\R^n)^*$ to $(\R^m)^*$) and we use the same notation for the extended map.  
We  define $\pi_a^*\colon (\R^n)^*\to (\R^m)^*$ by 
\[
\pi_a^*(u):=\pi^*(u)-\sum_{i=1}^ma_ie_i^*=\sum_{i=1}^m(\langle u,v_i\rangle-a_i)e_i^*.
\]
The map $\pi_a^*$ embeds $P$ into the positive orthant of $(\R^m)^*$.  The fiber product of 
$\pi_a^*$ and the (moment) map 
\begin{equation} \label{momPhi}
\Phi\colon \C^m\to (\R^m)^*
\end{equation}
sending $z=(z_1,\dots,z_m)$ to $\frac{1}{2}\sum_{i=1}^m|z_i|^2e_i^*$  
is 
\begin{equation} \label{fprod}
\{ (z,u)\in \C^m\times (\R^n)^*\mid \frac{1}{2}\sum_{i=1}^m|z_i|^2e_i^*=
\pi_a^*(u)\}.
\end{equation}
Since $\pi_a^*$ is injective and $\iota^*(\pi_a^*(u))=-\sum_{i=1}^ma_i\iota^*(e_i^*)$ for 
any $u$, the first projection from $\C^m\times (\R^n)^*$ onto $\C^m$ 
maps the fiber product \eqref{fprod} diffeomorphically onto 
\begin{equation} \label{ZP}
\mathcal{Z}_P:=\{ z\in\C^m\mid \sum_{i=1}^m(\frac{1}{2}|z_i|^2+a_i)\iota^*(e_i^*)=0\}.
\end{equation}
Note that 
\begin{equation} \label{ZPpull}
\mathcal{Z}_P=\Phi^{-1}(\pi_a^*(P)).
\end{equation} 

\begin{rema}
The manifold $\mathcal{Z}_P$ is called the \emph{moment-angle manifold} of $P$ 
and its topology is intensively studied in \cite{bu-pa02}.  It is also studied 
in \cite{bo-me06} from the viewpoint of real algebraic geometry.  
\end{rema}

We identify $\R/\Z$ with the unit circle $S^1$ of the complex numbers $\C$ through 
the exponential map $x\to \exp(2\pi\sqrt{-1}x)$ and set 
\[
T:=(S^1)^n. 
\]
The map $\pi_*$ in \eqref{exactV} induces an epimorphism 
$$\V\colon (S^1)^m\to T$$ 
and we make an identification 
\begin{equation} \label{Vident}
(S^1)^m/\Ker\V=T
\end{equation}
through the map $\V$. 
An element $c=(c_1,\dots,c_n)\in \Z^n$ defines a homomorphism 
\begin{equation} \label{lambda}
\lambda_c\colon S^1\to T
\end{equation}
sending $g$ to $(g^{c_1},\dots,g^{c_n})$ and we note that 
\begin{equation} \label{Vlambda}
\V(g_1,\dots,g_m)=\prod_{i=1}^m\lambda_{v_i}(g_i) \quad\text{for $(g_1,\dots,g_m)\in (S^1)^m$}.
\end{equation}

An element $b=(b_1,\dots,b_n)$ in  $\Z^n$ also defines a homomorphism 
\[
\chi^b\colon T\to S^1
\]
sending $(h_1,\dots,h_n)$ to $\prod_{i=1}^nh_i^{b_i}$.  Then we have 
\begin{equation} \label{comp}
(\chi^b\circ \lambda_c)(g)=g^{\langle b,c\rangle}\quad \text{for $g\in S^1$}
\end{equation} 
where $\langle b ,c \rangle=\sum_{i=1}^n b_ic_i$.  
Since the intersection of the kernels of $\chi^b\colon T\to S^1$ for all $b$ 
is trivial, it follows from \eqref{Vlambda} and \eqref{comp} that 
\begin{equation} \label{KerV}
\Ker\V=\{(g_1,\dots,g_m)\in (S^1)^m\mid \prod_{i=1}^mg_i^{\langle u,v_i\rangle}=1
\text{ for $\forall u\in \Z^n$}\}.
\end{equation}

\begin{rema}
The elements $b$ and $u$ above are taken from $\Z^n$ but we will see that  
it would be better to regard them as  elements of $(\Z^n)^*$ through the product $\langle\ ,\ \rangle$.  
\end{rema}

The group $(S^1)^m$ acts on $\C^m$ by componentwise multiplication and this action leaves 
$\mathcal{Z}_P$ invariant.  
The map $\Phi$ in \eqref{momPhi} induces a homeomorphism 
from the quotient $\C^m/(S^1)^m$ onto the positive orthant of $(\R^m)^*$ and 
$\Phi(\mathcal{Z}_P)=\pi_a^*(P)$.
The action of $(S^1)^m$ restricted to  $\Ker\V$ is free on 
$\mathcal{Z}_P$ and the quotient $\mathcal{Z}_P/\Ker\V$ is known to be the given $M$.  
The standard symplectic form 
$$\omega_0:=\frac{\sqrt{-1}}{2}\sum_{i=1}^m dz_i\wedge d\bar z_i$$ 
on $\C^m$ is invariant under the linear action of the unitary group $\U(m)$. 
The form $\omega_0$ descends to the given $\omega$ on $M$.  In fact, if 
\begin{equation*} \label{q}
q\colon \mathcal{Z}_P\to M=\mathcal{Z}_P/\Ker\V
\end{equation*}
denotes the quotient map, then $\omega$ satisfies 
\begin{equation} \label{omega0}
\omega_0|_{\mathcal{Z}_P}=q^*(\omega)
\end{equation}
and is uniquely determined by this identity, where the left-hand side denotes the restriction of $\omega_0$ to $\mathcal{Z}_P$. 
The action of $(S^1)^m$ on $\mathcal{Z}_P$ induces an action of $T=(S^1)^m/\Ker\V$ on 
$M=\mathcal{Z}_P/\Ker\V$ and this $T$-action on $M$ preserves the symplectic form $\omega$.   

As remarked before, the equation $\langle u,v_i\rangle=a_i$ defines the facet $P_i$ of $P$ for each $i=1,\dots,m$, and  
\[
\mathcal{Z}_{P_i}:=\Phi^{-1}(\pi_a^*(P_i))\quad\text{and}\quad M_i:=q(\mathcal{Z}_{P_i})
\]
are respectively closed smooth submanifolds of $\mathcal{Z}_P$ and $M$ of real codimension 2.  
We call $M_i$'s the \emph{characteristic submanifolds} of $M$. 
We see from \eqref{ZP} or \eqref{ZPpull} that 
$$\mathcal{Z}_{P_i}=\mathcal{Z}_P\cap \{z_i=0\}
$$  
and hence it follows from \eqref{lambda} and \eqref{Vlambda} that 

\begin{lemm} \label{vi}
The characteristic submanifold $M_i$ is fixed pointwise by the $S^1$-subgroup $\lambda_{v_i}(S^1)$ of $T$ for each $i$.
\end{lemm}

\section{Equivariant cohomology} \label{sect:2}

It is more convenient and natural to interpret the Delzant construction  
in terms of equivariant (co)homology.  We will discuss it and also recall some facts on equivariant 
cohomology in this section.  Recall that the equivariant homology and cohomology of a space $X$ with an action 
of the torus $T$ are respectively defined as 
\[
H_*^T(X):=H_*(ET\times_T X)\quad\text{and}\quad H^*_T(X):=H^*(ET\times_T X)
\]
where $ET\to BT=ET/T$ is a universal principal $T$-bundle and $ET\times_T X$ is the quotient 
of $ET\times X$ by the $T$-action given by 
\begin{equation} \label{balan}
t(e,x)= (et^{-1},tx) \quad\text{for $(e,x)\in ET\times X$ and $t\in T$.}
\end{equation}

Let $(M,\omega)$ be a symplectic toric manifold and let $M_i$'s $(i=1,\dots,m)$ be the 
characteristic submanifolds of $M$.  Since the $T$-action on $M$ preserves the symplectic form 
$\omega$ and $M_i$ is fixed pointwise under a circle subgroup of $T$ (see Lemma~\ref{vi}), 
the $\omega$ restricted 
to $M_i$ is again a symplectic form.  Therefore the form $\omega$ and its restriction to $M_i$ 
define orientations on $M$ and $M_i$. 
Since $M$ and $M_i$ are oriented and the inclusion map from $M_i$ to $M$ is  
equivariant, it defines an equivariant Gysin homomorphism 
\[
H_T^*(M_i)\to H^{*+2}_T(M)
\]
which raises the cohomological degree by 2 because the codimension of $M_i$ in $M$ is 2.   
We denote by $\tau_i$ the image of the unit element $1\in H^0_T(M_i)$ by the 
equivariant Gysin homomorphism.  
The cohomological degree of $\tau_i$ is 2.  
We may think of $\tau_i$ as the Poincar\'e dual of 
the cycle $M_i$ in the equivariant setting.  Since a cup product $\prod_{i\in I}\tau_i$ for 
$I\subset [m]:=\{1,\dots,m\}$ is the Poincar\'e dual of $\cap_{i\in I}M_i$, we see that 
\begin{equation} \label{if}
\text{$\prod_{i\in I}\tau_i=0$\quad if\quad $\cap_{i\in I}M_i=\emptyset$.}
\end{equation}   
It turns out that $H^*_T(M)$ is generated by $\tau_i$'s as a ring and that 
the relations in \eqref{if} are the only relations among $\tau_i$'s, i.e. we have 

\begin{lemm} \label{ring}
$H^*_T(M)=\Z[\tau_1,\dots,\tau_m]/(\prod_{i\in I}\tau_i\mid \cap_{i\in I}M_i=\emptyset)$ as rings. 
\end{lemm}

In particular, $\tau_i$'s are a free additive basis of 
$H^2_T(M)$ and the following easily follows from this fact. 

\begin{lemm} {\rm (see \cite[Lemma 1.5]{masu99} for example).} \label{lemm15}
Let $\pi\colon  ET\times_T M\to BT$ be the projection on the first factor.  Then 
for each $i=1,\dots,m$, there is a unique element $v_i\in H_2(BT)$ such that 
\begin{equation} \label{pi*}
\pi^*(u)=\sum_{i=1}^m\langle u,v_i\rangle\tau_i\quad\text{for any\ \ $u\in H^2(BT)$}
\end{equation}
where $\langle\ ,\ \rangle$ denotes the natural pairing between cohomology and homology. 
\end{lemm}

The Leray-Serre spectral sequence of the fibration 
\begin{equation} \label{fibr}
M\stackrel{\iota}\longrightarrow ET\times_TM\stackrel{\pi}\longrightarrow BT
\end{equation}
collapses because $H^{odd}(M)=H^{odd}(BT)=0$.  Therefore  
$H^*_T(M)=H^*(BT)\otimes H^*(M)$ as $H^*(BT)$-modules and hence 
$H^*(M)$ is the quotient of $H^*_T(M)$ by the ideal generated by $\pi^*(u)$ for 
$u\in H^2(BT)$. 
This together with Lemmas~\ref{ring} and~\ref{lemm15} implies the following well-known fact. 

\begin{prop} \label{DaJu}
We set 
\begin{equation} \label{mu}
\mu_i:=\iota^*(\tau_i) \in H^2(M). 
\end{equation}
Then $H^*(M)$ is the quotient of a polynomial ring $\Z[\mu_1,\dots,\mu_m]$ by the 
ideal generated by the following two types of elements:
\begin{enumerate}
\item $\prod_{i\in I}\mu_i$ for $I\subset [m]$ with $\cap_{i\in I}M_i=\emptyset$.
\item $\sum_{i=1}^m\langle u,v_i\rangle\mu_i$ for $u\in H^2(BT)$.
\end{enumerate}
\end{prop}

A homomorphism $f\colon S^1\to T$ induces a continuous map $Bf\colon BS^1\to BT$.  
We fix a generator $\kappa$ of $H_2(BS^1)\cong\Z$.  Then the correspondence $f\to (Bf)_*(\kappa)$ 
defines an isomorphism
\begin{equation} \label{hom} 
\Hom(S^1,T)\cong H_2(BT)
\end{equation}
and we denote by $\lambda_v$ the element of $\Hom(S^1,T)$ corresponding to $v\in H_2(BT)$.  
The identity in Lemma~\ref{lemm15} implies the following. 

\begin{lemm}{\rm (see \cite[Lemma 1.10]{masu99} for example).} 
For the elements $v_i\in H_2(BT)$  $(i=1,\dots,m)$ defined in Lemma~\ref{lemm15}, 
$\lambda_{v_i}(S^1)$ is the circle subgroup of $T$ which fixes $M_i$ pointwise. 
\end{lemm}

This lemma corresponds to Lemma~\ref{vi}.  More precisely, one can see 
that the $v_i$'s defined in Lemma~\ref{lemm15} can be identified with the $v_i$'s 
in Section~\ref{sect:1} through an identification 
\[
H_2(BT)=\Z^n,
\] 
(see \cite{masu99} for example).  Taking the dual of this identification, we obtain an identification  
\[
H^2(BT)=(\Z^n)^*.
\]
Then \eqref{P} can be rewritten as 
\begin{equation} \label{P2}
P=\{u\in H^2(BT;\R)\mid \langle u,v_i\rangle\ge a_i\quad(i=1,\dots,m)\}
\end{equation}
where $\langle\ ,\ \rangle$ denotes the natural pairing between cohomology and 
homology as before. 

\begin{lemm} \label{iden}
The exact sequences in \eqref{exactV} and 
\eqref{exactV*} can be regarded as exact sequences derived from the fibration 
\eqref{fibr}, namely, \eqref{exactV} can be regarded as 
\begin{equation} \label{exact}
0\to H_2(M)\stackrel{\iota_*}\longrightarrow H_2^T(M)
\stackrel{\pi_*}\longrightarrow H_2(BT)\to 0
\end{equation}
and \eqref{exactV*} as 
\begin{equation} \label{exact*}
0\leftarrow H^2(M)\stackrel{\iota^*}\longleftarrow H_T^2(M)\stackrel{\pi^*}\longleftarrow
H^2(BT)\leftarrow 0.
\end{equation}
\end{lemm}

\begin{proof}
If we identify $H_T^2(M)$ with $(\Z^m)^*$ through the identification of $\tau_i$ with $e_i^*$ for $i=1,\dots,m$, 
then \eqref{pi*} agrees with \eqref{V} and this implies the lemma. 
\end{proof}

Through the identifications in Lemma~\ref{iden}, \eqref{ZP} turns into 
\begin{equation} \label{ZP2}
\mathcal{Z}_P=\{ z\in\C^m\mid \sum_{i=1}^m(\frac{1}{2}|z_i|^2+a_i)\mu_i=0\}
\end{equation}
where $\mu_i$'s are the elements of $H^2(M)$ defined in \eqref{mu}.  


\section{Roots of a compact Lie subgroup of $\Sympl(M,\omega)$} \label{sect:3}

If $g\in \Sympl(M,\omega)$ normalizes the torus $T$, then $\rhog$ defined by 
\begin{equation} \label{3rhog}
\text{$\rhog(t):=gtg^{-1}$}
\end{equation}
is a group automorphism of $T$ and the diffeomorphism $g$ of $M$ is $\rhog$-equivariant.  
Let $E\rhog$ be a homeomorphism of $ET$ induced from $\rhog$.  It is 
$\rhog$-equivariant, i.e.  $E\rhog(et)=E\rhog(e)\rhog(t)$ for $e\in ET$ and $t\in T$.  
Therefore, a homeomorphism of $ET\times M$ sending $(e,x)$ to $(E\rhog(e), gx)$ 
is $\rhog$-equivariant and  
induces a homeomorphism of $ET\times_T M$.  Hence we obtain a ring 
automorphism of $H^*_T(M)$, denoted by $g^*$, which preserves the subalgebra $\pi^*(H^*(BT))$. 
It easily follows from the definition of $g^*$ that  
\begin{equation} \label{gpai}
g^*\circ \pi^*=\pi^*\circ {\rhog}^*\quad\text{on $H^*(BT)$}
\end{equation} 
where ${\rhog}^*$ is an automorphism of $H^*(BT)$ induced from $\rhog$. 

Since the diffeomorphism $g$ of $M$ is $\rhog$-equivariant, it permutes the 
characteristic submanifolds $M_i$'s.  Moreover, since $g$ preserves the form $\omega$, 
it preserves the orientations on $M$ and $M_i$'s induced from $\omega$.  These imply that  
there is a permutation $\sigma$ on $[m]$ such that 
\begin{equation} \label{tperm}
g^*(\tau_i)=\tau_{\sigma(i)} \quad\text{for any $i$}.
\end{equation}

With these understood 

\begin{lemm} \label{lemm:rhov}
Let $\rhog_*$ be an automorphism of $H_*(BT)$ induced from $\rhog$.  Then 
$\rhog_*(v_{\sigma(i)})=v_i$ for any $i$. 
\end{lemm}

\begin{proof}
Sending the identity \eqref{pi*} by $g^*$, it follows from \eqref{gpai} and \eqref{tperm} that we have  
\begin{equation} \label{comp1}
\pi^*(\rhog^*(u))=g^*(\pi^*(u))=\sum_{i=1}^m\langle u,v_i\rangle g^*(\tau_i)=\sum_{i=1}^m \langle u,v_i\rangle\tau_{\sigma(i)},
\end{equation}
while it follows from \eqref{pi*} applied to $\rhog^*(u)$ 
instead of $u$ that we have 
\begin{equation} \label{comp2}
\begin{split}
\pi^*(\rhog^*(u))&=\sum_{i=1}^m \langle \rhog^*(u),v_i\rangle\tau_i 
=\sum_{i=1}^m \langle \rhog^*(u),v_{\sigma(i)}\rangle\tau_{\sigma(i)}\\ 
&=\sum_{i=1}^m \langle u,\rhog_*(v_{\sigma(i)})\rangle\tau_{\sigma(i)}.
\end{split}
\end{equation}
Comparing \eqref{comp1} with \eqref{comp2} and noting that 
$\tau_{\sigma(i)}$'s are free over $\Z$, we obtain  
\[
\langle u,v_i\rangle =\langle u,\rhog_*(v_{\sigma(i)})\rangle \quad\text{for any $i$},
\]
but since this identity holds for any $u\in H^2(BT)$, the desired identity in the lemma follows. 
\end{proof} 

The following lemma is due to M. Wiemeler and will play a key role in the subsequent argument. 

\begin{lemm}{\rm (\cite[Lemma 2.1]{wiem08})}. \label{W}
If $g$ induces the identity on $H^2(M)$ and $\rhog^*$ is a reflection on $H^2(BT)$, 
then the $\sigma$ in \eqref{tperm} permutes exactly two elements in $[m]$ and fixes the others. 
\end{lemm}

\begin{proof}
Since $\rhog^*$ is a reflection, its trace is $n-2$.  On the other hand, since $g$ induces the identity 
on $H^2(M)$ by assumption, the trace of $g^*$ on $H^2_T(M)$ must be $m-2$ by \eqref{exact*}.  
However, $H^2_T(M)$ is freely generated by $\tau_i$'s over $\Z$ and $g^*$ permutes the generators by \eqref{tperm},  
so the lemma follows. 
\end{proof}

Dualizing the isomorphism \eqref{hom}, we obtain an isomorphism 
\begin{equation} \label{cohom}
\Hom(T,S^1)\cong H^2(BT).
\end{equation}
For $u\in H^2(BT)$, we denote by $\chi^u\in \Hom(T,S^1)$ 
the element corresponding to $u$ through the isomorphism \eqref{cohom}. 

Now we take a compact Lie subgroup $G$ of $\Sympl(M,\omega)$ containing $T$ 
and denote by $\G0$ the identity component of $G$.  As remarked in the Introduction, 
the torus $T$ is a maximal torus of $G$. 

\begin{defi}
A root of $G$ is a non-zero weight of the adjoint representation of $T$ 
on $\g\otimes\C$, where $\g$ denotes the Lie algebra of $G$.  
We think of a root of $G$ as an element of $H^2(BT)$ through the isomorphism 
\eqref{cohom} and denote by $\RG$ the root system of $G$, that is the set of roots of $G$.
Needless to say, $\RG$ depends only on the identity component $\G0$. 
\end{defi}

For $\alpha\in \RG$, we denote by $T_\alpha$ the identity component of the kernel of 
$\chi^\alpha\colon T\to S^1$.   Since $\alpha$ is non-zero, $T_\alpha$ is a codimension 1 
subtorus of $T$.  
Let $\G0_\alpha$ be the identity component of the subgroup of $\G0$ 
which commutes with $T_\alpha$.  
The group $N_{\G0_\alpha}(T)/T$ is of order two and let $g\in N_{\G0_\alpha}(T)$ be 
a representative of the non-trivial element in $N_{\G0_\alpha}(T)/T$. 
The automorphism $\rho_g$ of $T$ 
is independent of the choice of the representative $g$, so we may denote it by $\rho_\alpha$. 
It is of order two, its fixed point set contains the codimension 1 subtorus $T_\alpha$ 
and $\rho_\alpha^*(\alpha)=-\alpha$.  
We note that $\rho_\alpha^*$ is the Weyl group action associated with $\alpha\in \RG$. 

Similarly, ${\rho_\alpha}_*$ is a reflection on $H_2(BT)$ and we note that
\begin{equation} \label{eqalpha}
\text{$\Fix({\rho_\alpha}_*)=H_2(BT_\alpha)=\Ker\alpha$. }
\end{equation}

\begin{lemm} \label{3}
For $\alpha\in\RG$, there are $i, j\in [m]$ such that 
\begin{equation} \label{eq31}
\langle \alpha,v_i\rangle=-\langle \alpha, v_j\rangle(\not=0)\ 
\text{ and }\ \langle \alpha,v_k\rangle=0 \ \text{for any $k\not=i,j$}.
\end{equation}
\end{lemm}

\begin{proof}
Let $g\in N_{\G0_\alpha}(T)$ be a representative of the non-trivial element in $N_{\G0_\alpha}(T)/T$. 
Since $g$ is in $\G0$, it is homotopic to the identity so that $g$ induces the identity on $H^2(M)$.  
Moreover, $\rho_g^*=\rho_\alpha^*$ is a reflection as observed above.  Therefore  
there are $i, j\in [m]$ such that 
\[
g^*(\tau_i)=\tau_j,\quad g^*(\tau_j)=\tau_i,\quad g^*(\tau_k)=\tau_k \ \text
{for any $k\not=i,j$}
\]
by Lemma~\ref{W}.  It follows from Lemma~\ref{lemm:rhov} that 
\begin{equation} \label{rhov}
{\rho_\alpha}_*(v_i)=v_j,\quad {\rho_\alpha}_*(v_j)=v_i, \quad 
{\rho_\alpha}_*(v_k)=v_k \ \text{for any $k\not=i,j$}
\end{equation}
and hence $\langle\alpha,v_k\rangle=0$ for $k\not=i,j$ by \eqref{eqalpha}. 
Finally, since $\rho_\alpha^*(\alpha)=-\alpha$, we have
\[
\langle \alpha,v_i\rangle=-\langle \rho_\alpha^*(\alpha),v_i\rangle=-\langle \alpha,
{\rho_\alpha}_*(v_i)\rangle=-\langle\alpha,v_j\rangle,
\]
proving the lemma. 
\end{proof}

\begin{lemm} \label{4}
For $\alpha, \beta\in \RG$ we have  
\begin{equation*}
\rho_\alpha^*(\beta)=\beta-\frac{\langle\beta,v_i\rangle-\langle\beta,v_j\rangle}
{\langle\alpha,v_i\rangle}\alpha=\beta-\frac{\langle\beta,v_j\rangle-\langle\beta,v_i\rangle}
{\langle\alpha,v_j\rangle}\alpha. 
\end{equation*}
\end{lemm}

\begin{proof}
It follows from \eqref{rhov} that
\[
\begin{split}
\langle \rho_\alpha^*(\beta)-\beta,v_i\rangle&=\langle \beta, {\rho_\alpha}_*(v_i)\rangle -\langle \beta,v_i\rangle=
\langle\beta,v_j\rangle-\langle\beta,v_i\rangle \\
\langle \rho_\alpha^*(\beta)-\beta,v_j\rangle&=
\langle \beta, {\rho_\alpha}_*(v_j)\rangle -\langle \beta,v_j\rangle=
\langle\beta,v_i\rangle-\langle\beta,v_j\rangle\\
\langle \rho_\alpha^*(\beta)-\beta,v_k\rangle&=\langle \beta, {\rho_\alpha}_*(v_k)\rangle -\langle \beta,v_k\rangle
=0\quad\text{for $k\not=i,j$.} 
\end{split}
\]
This together with \eqref{eq31} shows that 
the three terms in the lemma take a same value on each $v_\ell$.  Since $v_\ell$'s span 
$H_2(BT)$, the desired identity in the lemma follows. 
\end{proof}

Let $a_{\beta,\alpha}$ be the constant defined by 
\begin{equation} \label{rhoal}
\rho_\alpha^*(\beta)=\beta-a_{\beta,\alpha}\alpha.
\end{equation}
By Lemma~\ref{4} we have 
\begin{equation}  \label{eq41}
a_{\beta,\alpha}=\frac{\langle\beta,v_i\rangle-\langle\beta,v_j\rangle}
{\langle\alpha,v_i\rangle}=\frac{\langle\beta,v_j\rangle-\langle\beta,v_i\rangle}
{\langle\alpha,v_j\rangle}.
\end{equation}
The following is well-known (see \cite[9.4]{hump70} but $a_{\beta,\alpha}$ is denoted 
$\langle\beta,\alpha\rangle$ in the book): 
\begin{enumerate}
\item $a_{\beta,\alpha}$ is an integer, 
\item $a_{\beta,\alpha}\not=0$ if and only if $a_{\alpha,\beta}\not=0$, 
\item $0\le a_{\beta,\alpha}a_{\alpha,\beta}\le 3$ if $\beta\not=\pm\alpha$
\end{enumerate}
We set 
$$N_\alpha:=|\langle\alpha,v_i\rangle|=|\langle\alpha,v_j\rangle|.
$$ 
We say that $\alpha$ and $\beta$ are \emph{joined} if $\rho_\alpha^*(\beta)\not=\beta$ (i.e. $a_{\beta,\alpha}\not=0$).  
By (2) above, $\rho_\alpha^*(\beta)\not=\beta$ if and only if $\rho^*_\beta(\alpha)\not=\alpha$. 

\begin{lemm} \label{5}
If $\alpha$ and $\beta$ are joined, then $N_\alpha=N_\beta$.  
\end{lemm}

\begin{proof}
Since $N_{-\alpha}=N_\alpha$, we may assume $\beta\not=\pm\alpha$. 
Then $a_{\beta,\alpha}$ is a non-zero integer, so it follows from \eqref{eq41} that 
$|a_{\beta,\alpha}|\ge 2$ if $N_\alpha\not=N_\beta$ (note that since $\beta\not=\pm\alpha$, either 
$\langle \beta,v_i\rangle$ or $\langle\beta,v_j\rangle$ is 0 by Lemma~\ref{3}). 
Changing the role of $\alpha$ and $\beta$, we also have that $|a_{\alpha,\beta}|\ge 2$ 
if $N_\alpha\not=N_\beta$. 
But this contradicts the above fact (3) that 
$0\le a_{\beta,\alpha}a_{\alpha,\beta}\le 3$ if $\beta\not=\pm\alpha$.
\end{proof}

$\RG$ decomposes into a direct sum of irreducible root systems.  
Since we are concerned with the isomorphism type of $\RG$ as a root system,   
we may assume that $N_\alpha=1$ for any $\alpha$ by Lemma~\ref{5}.

\section{The root system of a moment polytope} \label{sect:4}

Remember that the correspondence from symplectic toric manifolds to their moment polytopes 
(which are non-singular) is one-to-one.  
Motivated by the observation made in Section~\ref{sect:3}, we make the following definition.

\begin{defi} For a non-singular polytope $P$ described in  \eqref{P2}, we define  
\[
\begin{split}
\RM:=\{\alpha\in H^2(BT)\mid &
\quad\langle \alpha,v_i\rangle=1,\ \langle \alpha,v_j\rangle=-1\quad \text{for some $i,j$},\\ 
&\qquad\quad\text{and}\quad\langle\alpha,v_k\rangle=0\quad \text{for $k\not=i,j$}\}, 
\end{split}
\]
and call it the \emph{root system} of $P$. (It will be proved below that $\RM$ 
is actually a root system.) 
\end{defi}

\begin{rema}
The root system $\RM$ depends only on the $v_i$'s and not on the constants $a_i$'s used 
to describe the moment polytope $P$ in \eqref{P2}. 
 \end{rema} 

\begin{exam}
We identify $H_2(BT)$ with $\Z^n$ and denote by $\{e_i\}_{i=1}^n$ the standard basis of $\Z^n$ 
and by $\{e_i^*\}_{i=1}^n$ the basis of $(\Z^n)^*$ dual to $\{e_i\}_{i=1}^n$. 
Remember that $m$ is the number of facets of $P$. 

(1) Let $m=n+1$ and take $v_i=e_i$ for $1\le i\le n$ and $v_{n+1}=-\sum_{i=1}^ne_i$.  Then 
\[
\RM=\{  \pm e_i^*\ (1\le i\le n),\quad \pm(e_i^*-e_j^*)\ (1\le i<j\le n)\}
\]
and this is a root system of type $A_n$. Note that the manifold $M$ is the complex projective space $\C P^n$ 
of complex dimension $n$ in this case. 

(2) Let $n=2$, $m=4$ and take $v_1=e_1,\ v_2=e_2,\ v_3=-e_1+ae_2$ and $v_4=-e_2$ where $a$ is 
an arbitrary integer. If $a=0$, then $P$ is a rectangle and 
\[
\RM=\{\pm e_1^*, \ \pm e_2^* \}
\]
which is of type $A_1\times A_1$, and if $a\not=0$, then $P$ is a trapezoid with two right angle corners 
and  
\[
\RM=\{\pm e_1^*\}
\]
which is of type $A_1$.  Note that the manifold $M$ is a Hirzebruch surface in this case.  

(3) If $n=2$ and $m\ge 5$, then one easily checks that $R(P)$ is empty. 
\end{exam}

For $\alpha\in \RM$ with $\langle \alpha,v_i\rangle=1$ and 
$\langle \alpha,v_j\rangle=-1$, we define a reflection $r_\alpha$ on $H_2(BT)$ by
\[
r_\alpha(v):=v-\langle \alpha,v\rangle(v_i-v_j)
\]
which interchanges $v_i$ and $v_j$ and fixes $v_k$'s for $k\not=i,j$, and 
define its dual reflection $r_\alpha^\vee$ on $H^2(BT)$ by
\begin{equation*} \label{eq62}
\langle r_\alpha^\vee(\beta),v\rangle:=\langle\beta,r_\alpha(v)\rangle=
\langle \beta,v-\langle \alpha,v\rangle(v_i-v_j)\rangle. 
\end{equation*}
This shows that 
\begin{equation} \label{eq82}
r_\alpha^\vee(\beta)=\beta-(\langle\beta,v_i\rangle-\langle\beta,v_j\rangle)\alpha.
\end{equation}
In particular, $r_\alpha^\vee(\pm\alpha)=\mp\alpha$. 
One can easily check that $r_\alpha^\vee$ preserves $\RM$. 
Comparing Lemma~\ref{4} with \eqref{eq82} and noting that we may assume $N_\alpha=1$ for 
any $\alpha$, we see that ${\rho_\alpha}^*$ agrees with $r_\alpha^\vee$ and hence we obtain the following. 

\begin{prop} \label{root} 
$\RM$ is a root system and if $P$ is the moment polytope associated with a symplectic toric manifold $(M,\omega)$ and 
$G$ is a compact Lie subgroup of $\Sympl(M,\omega)$ containing the torus $T$, then 
the root system $\RG$ of $G$ is a subsystem of $\RM$. 
\end{prop}

We define a symmetric scalar product $(\ ,\ )$ on $H^2(BT)$ by 
\begin{equation} \label{bilinear}
(\beta,\gamma):=\sum_{\ell=1}^m\langle \beta,v_\ell\rangle\langle\gamma,v_\ell\rangle.
\end{equation}
One easily sees from the definition of $\RM$ that 
\begin{equation}  \label{aa}
\begin{split}
(\alpha,\alpha)&=2 \qquad\text{for $\alpha\in \RM$,}\\
(\alpha,\beta)&=0 \text{ or }\pm 1 \quad\ \text{for $\alpha,\beta\in \RM$ with $\beta\not=\pm\alpha$.}
\end{split}  
\end{equation}
The group generated by the reflections $r_\alpha^\vee$ $(\alpha\in \RM)$ 
is called the \emph{Weyl group} of $\RM$.  

\begin{lemm} \label{winva}
The scalar product $(\ ,\ )$ is invariant under the Weyl group of $\RM$. 
\end{lemm}

\begin{proof}
It suffices to check that $(r_\alpha^\vee(\beta),r_\alpha^\vee(\gamma))=(\beta,\gamma)$ 
for $\alpha\in \RM$ and $\beta,\gamma\in H^2(BT)$.  By definition there are $i,j\in [m]$ such that 
$\langle\alpha,v_i\rangle=1$, $\langle\alpha,v_j\rangle=-1$ and $\langle \alpha,v_k\rangle=0$ 
for $k\not=i,j$. 
We denote $\langle \eta,v_\ell\rangle$ by $\eta_\ell$ for $\eta\in H^2(BT)$. 
Then since $(\alpha,\eta)=\eta_i-\eta_j$, it follows from \eqref{eq82} that 
\[
\begin{split}
&(r_\alpha^\vee(\beta),r_\alpha^\vee(\gamma))
=(\beta-(\beta_i-\beta_j)\alpha,\gamma-(\gamma_i-\gamma_j)\alpha)\\
=&(\beta,\gamma)-(\beta_i-\beta_j)(\alpha,\gamma)-(\gamma_i-\gamma_j)(\beta,\alpha)+
(\beta_i-\beta_j)(\gamma_i-\gamma_j)(\alpha,\alpha)\\
=&(\beta,\gamma)-(\beta_i-\beta_j)(\gamma_i-\gamma_j)-(\gamma_i-\gamma_j)(\beta_i-\beta_j)
+2(\beta_i-\beta_j)(\gamma_i-\gamma_j)\\
=&(\beta,\gamma),
\end{split}
\]
proving the lemma. 
\end{proof}

For $\alpha,\beta\in \RM$ we define an integer $a_{\beta,\alpha}$ by 
\begin{equation} \label{rab}
r_\alpha^\vee(\beta)=\beta-a_{\beta,\alpha}\alpha.  
\end{equation}
similarly to \eqref{rhoal}.  
If  $\langle\alpha,v_i\rangle=1$ and $\langle\alpha,v_j\rangle=-1$, then    
\begin{equation} \label{eqn:aab}
a_{\beta,\alpha}=\langle\beta,v_i\rangle-\langle\beta,v_j\rangle
\end{equation}
by \eqref{eq82}. 
Another description of $a_{\beta,\alpha}$ is the following. 

\begin{lemm} \label{lemm:symm}
$a_{\beta,\alpha}=(\alpha,\beta)$ for $\alpha,\beta\in \RM$.  
In particular, $a_{\beta,\alpha}=a_{\alpha,\beta}$. 
\end{lemm}

\begin{proof}
Since $r_\alpha^\vee(\alpha)=-\alpha$ and $r_\alpha^\vee$ is of order 2, 
it follows from Lemma~\ref{winva} that 
\[
(\alpha,r_\alpha^\vee(\beta)+\beta)=(r_\alpha^\vee(\alpha),r_\alpha^\vee(r_\alpha^\vee(\beta)+\beta))=
(-\alpha,\beta+r_\alpha^\vee(\beta))
\]
and hence $(\alpha,r_\alpha^\vee(\beta)+\beta)=0$.  This together with 
\eqref{aa} and \eqref{rab} implies the lemma. 
\end{proof}

\begin{theo} \label{typeA}
Any irreducible subsystem of $\RM$ is of type A. 
\end{theo}

\begin{proof}
Let $\Phi$ be an irreducible subsystem of $\RM$. 
The Cartan matrix $C(\Phi)$ of $\Phi$ is $(a_{\beta,\alpha})$ where $\alpha$ and $\beta$ run over 
elements in a basis of $\Phi$.  The diagonal entiries of $C(\Phi)$ are all 2 by \eqref{aa} 
and $C(\Phi)$ is symmetric by Lemma~\ref{lemm:symm}.  Therefore, $\Phi$ must be either 
of type A, D or E (see \cite[p.59]{hump70}). 

Suppose that $\Phi$ is of type D or E.  
Then there are elements $\alpha,\beta,\gamma,\delta$ in the basis of $\Phi$ such that 
\begin{equation} \label{-13}
a_{\beta,\alpha}=a_{\gamma,\alpha}=a_{\delta,\alpha}=-1.
\end{equation}
As before, let $v_i,v_j$ be the elements such that $\langle\alpha,v_i\rangle=1$ and 
$\langle\alpha,v_j\rangle=-1$.  
It follows from \eqref{eqn:aab} and \eqref{-13} that 
the values which $\beta,\gamma,\delta$ take on $v_i$ and $v_j$ must be 
either $(-1,0)$ or $(0,1)$. 
Therefore two of $\beta,\gamma,\delta$, say $\beta$ and $\gamma$, 
must take the same values on $v_i$ and $v_j$, say $(0,1)$. (The same argument 
below will work for $(-1,0)$.) 
Let $v_k$ be the other element on which $\beta$ takes a non-zero value. 
Then $r_\beta(v_j)=v_k$ and since $\langle\beta,v_j\rangle=1$, we have 
$\langle \beta,v_k\rangle=-1$. Moreover, since $\langle \gamma,v_j\rangle=1$ and 
$\beta\not=\gamma$, we have $\langle \gamma,v_k\rangle=0$.  
Therefore 
\[
\begin{split}
\langle r_\beta^\vee(\gamma)-\gamma,v_j\rangle&=\langle \gamma,r_\beta(v_j)\rangle-
\langle \gamma,v_j\rangle=\langle \gamma, v_k\rangle
-\langle\gamma,v_j\rangle =-1 \\
\langle r_\beta^\vee(\gamma)-\gamma,v_k\rangle&=
\langle \gamma,r_\beta(v_k)\rangle-\langle \gamma,v_k\rangle=
\langle \gamma,v_j\rangle-\langle\gamma,v_k\rangle=1.
\end{split}
\]
Since $\beta$ takes 1 on $v_j$ and $-1$ on $v_k$, the above shows that 
$r_\beta^\vee(\gamma)-\gamma=-\beta$
and hence $a_{\gamma,\beta}=1$ by \eqref{rab}.  However $a_{\gamma,\beta}$ must be non-positive because 
$\beta$ and $\gamma$ are in the basis of $\Phi$ and $\beta\not=\gamma$.  This is a contradiction. 
Thus $\Phi$ is neither of type D nor E and hence of type A. 
\end{proof}

We conclude this section with the following corollary which follows from 
Proposition~\ref{root} and Theorem~\ref{typeA}. 

\begin{coro} \label{GtypeA}
If $G$ is a compact Lie subgroup of $\Sympl(M,\omega)$ containing the torus $T$,  
then any irreducible factor of $\RG$ is of type A. 
\end{coro}

\section{Connected maximal compact Lie subgroup of $\Sympl(M,\omega)$} \label{sect:5}

In this section we shall observe that the equality $\RG=\RM$ is attained 
for some compact \emph{connected} Lie subgroup $G$ of $\Sympl(M,\omega)$.  

As discussed in Section~\ref{sect:1}, we may think of $M$ as $\mathcal{Z}_P/\Ker\V$ and 
$\omega$ as the form induced from the standard form $\omega_0$ on $\C^m$, where 
\begin{equation} \label{5ZP2}
\mathcal{Z}_P=\{ z\in\C^m\mid \sum_{i=1}^m(\frac{1}{2}|z_i|^2+a_i)\mu_i=0\}
\end{equation}
from \eqref{ZP2} and 
\begin{equation} \label{5KerV}
\Ker\V=\{(g_1,\dots,g_m)\in (S^1)^m\mid \prod_{i=1}^mg_i^{\langle u,v_i\rangle}=1 \text{ for $\forall u
\in H^2(BT)$}\}
\end{equation}
from \eqref{KerV} through the identification $H^2(BT)=\Z^n$ discussed in Section~\ref{sect:2}.  

\begin{lemm} \label{mug}
Let $\alpha$ be an element of $\RM$ such that $\langle \alpha,v_i\rangle=1$, $\langle\alpha,v_j\rangle=-1$ and 
$\langle\alpha,v_k\rangle=0$ for $k\not=i,j$.  Then 
$\mu_i=\mu_j$ in \eqref{5ZP2} and $g_i=g_j$ in \eqref{5KerV}.
\end{lemm}

\begin{proof}
It follows from \eqref{pi*} that $\pi^*(\alpha)=\tau_i-\tau_j$.  Applying $\iota^*$ to the both sides of this 
identity, we get the former identity in the lemma because $\iota^*\circ\pi^*=0$ by \eqref{exact*} 
and $\mu_\ell=\iota^*(\tau_\ell)$ by \eqref{mu}.  
If we take the $\alpha$ as $u$ in \eqref{5KerV}, 
then the condition $\prod_{i=1}^mg_i^{\langle u,v_i\rangle}=1$ reduces to $g_ig_j^{-1}=1$ and this proves the 
latter statement in the lemma. 
\end{proof}

The purpose of this section is to prove the following. 

\begin{prop} \label{mconn}
There is a closed connected subgroup $\tilde G$ of the unitary group $\U(m)$ which leaves $\mathcal{Z}_P$ invariant 
and contains $\Ker\V$ in its center (so that the action of $\tilde G$ induces an effective action of 
$\tilde G/\Ker\V$ on $M=\mathcal{Z}_P/\Ker\V$) and $\Delta(\tilde G/\Ker\V)=\RM$.
\end{prop}

\begin{proof}
Let $\Phi$ be an irreducible factor of $\RM$.  It is of type A by Theorem~\ref{typeA}. 
Suppose that the rank of $\Phi$ is $r-1$. Then it follows from Lemma~\ref{mug} that 
there is a subset $I(\Phi):=\{i_1,\dots,i_{r}\}$ of $[m]$ such that 
$\mu_{i_1}=\dots=\mu_{i_{r}}$ and $g_{i_1}=\dots=g_{i_r}$ for $g=(g_1,\dots,g_m)\in\Ker\V$. 
Therefore the action of $\U(m)$ on $\C^m$ restricted to the subgroup  
\[
\U(\Phi):=\{(x_{ij})\in \U(m)\mid x_{ij}=\delta_{ij} \ \text{unless both $i$ and $j$ are in $I(\Phi)$}\}, 
\]
where $\delta_{ij}=1$ if $i=j$ and $0$ otherwise, leaves $\mathcal{Z}_P$ invariant and $\U(\Phi)$ 
commutes with $\Ker\V$.  
We note that the root system of $\U(\Phi)$ is (isomorphic to) $\Phi$. 

Now we decompose $\RM$ into sum of irreducible factors $\Phi_1,\dots,\Phi_s$.  Then the subsets $I(\Phi_1),\dots,I(\Phi_s)$ 
of $[m]$ are disjoint. 
We consider the subgroup $\tilde G$ of $\U(m)$ generated by $\prod_{i=1}^s\U(\Phi_i)$ and (the diagonal subgroup) $(S^1)^m$.  
Since $(S^1)^m$ contains $\Ker\V$, so does $\tilde G$.  
It follows from the observation above that $\tilde G$ commutes with $\Ker\V$ and the action of $\U(m)$ on $\C^m$ 
restricted to $\tilde G$ leaves $\mathcal{Z}_P$ invariant so that the action descends to an effective action of $\tilde G/\Ker\V$ on $M$. 
Since the action of $\tilde G$ on $\mathcal{Z}_P$ preserves the standard form $\omega_0|_{\mathcal{Z}_P}$, 
the induced action of $\tilde G/\Ker\V$ on $M$ preserves the form $\omega$ on $M$, see \eqref{omega0}.  
Finally, $\Delta(\tilde G/\Ker\V)=\RM$ by construction.  
\end{proof}

\section{Automorphisms of a moment polytope} \label{sect:6}

Since the dual of the Lie algebra of $T$ can be naturally identified with 
$H^2(BT;\R)$,  we think that the moment map $\mu$ associated 
with $(M,\omega)$ takes values in $H^2(BT;\R)$ and $P=\mu(M)$. 
When $g\in \Sympl(M,\omega)$ 
normalizes the torus $T$, we associated the $\rhog$ of $\Aut(T)$ to $g$ in Section~\ref{sect:3}, 
where 
\begin{equation} \label{defrhog}
\text{$\rhog(t)=gtg^{-1}$\quad for $t\in T$}
\end{equation} 
and $\Aut(T)$ denotes the group of automorphisms of $T$. 
The moment map associated to $(M,\omega)$ with the $T$-action twisted by $\rhog$ is 
given by $\rhog^*\circ\mu$, so the image of $M$ by the map is $\rhog^*(P)$. 
Since $g$ preserves the form $\omega$, the images of $M$ by $\mu$ and $\rhog^*\circ\mu$
are congruent modulo parallel translations in $H^2(BT;\R)$. 
Motivated by this observation, we define 
\[
\Aut(P):=\{\rho\in \Aut(T)\mid \rho^*(P)\equiv P\}
\]
where $\equiv$ denotes congruence modulo parallel translations in $H^2(BT;\R)$. 

\begin{rema} \label{rema2}
As remarked before, the root system $\RM$ depends only on the $v_i$'s and not on 
the constants $a_i$'s used to define the moment polytope $P$ in \eqref{P} or \eqref{P2}.  
However $\Aut(P)$ actually depends on the $a_i$'s.  
For instance, $\Aut(P)$ for a square $P$ is (a dihedral group) of order 8 while $\Aut(P)$ 
for a (non-square) rectangle $P$ is of order 4.
\end{rema}

The correspondence $g\to \rhog$ defines a homomorphism 
\begin{equation} \label{Phi}
\rhoP\colon \NGT\to \Aut(P)
\end{equation}
where $G$ is any subgroup of $\Sympl(M,\omega)$ containing $T$ (e.g. $G$ may be 
the entire group $\Sympl(M,\omega)$) and $\NGT$ denotes the normalizer of $T$ in $G$. 
If $g\in T$, then $\rhog$ is the identity; so $T$ is in the kernel of $\rhoP$. 

\begin{lemm} \label{lemm2}
If $G$ is a compact Lie subgroup of $\Sympl(M,\omega)$ containing the torus $T$, then 
the kernel of $\rhoP$ is exactly $T$.
\end{lemm}

\begin{proof} 
We note that $g\in \NGT$ permutes the characteristic submanifolds $M_i$'s of $M$.   
Suppose that $g\in \NGT$ is in the kernel of $\rhoP$.  
Then $g$ maps $M_i$ to itself for each $i$.  Let $x$ be a $T$-fixed point in $M$.  
Then $x=\bigcap_{i\in I}M_i$ for some $I\in [m]$ with cardinality $n$, so that $x$ is fixed by $g$.  
We decompose the tangent space $\tau_xM$ of $M$ at $x$ into 
\[
\tau_xM=\bigoplus_{i\in I}\tau_xM/\tau_xM_i. 
\]
The differential 
$dg\colon \tau_xM\to \tau_xM$ preserves each real 2-dimensional eigenspace $\tau_xM/\tau_xM_i$ since 
$g$ fixes $x$ and maps $M_i$ to itself for each $i$.  The symplectic form $\omega$ determines an orientation on $\tau_xM/\tau_xM_i$ 
for each $i$ and $dg$ preserves the orientation on $\tau_xM/\tau_xM_i$ since $g$ preserves the form $\omega$.  

Since $G$ is compact, there exists a  $G$-invariant Riemannian metric on $M$ 
so that we may assume that $dg$ is an orthogonal transformation on $\tau_xM/\tau_xM_i$ but 
since $dg$ preserves the orientation on it, $dg$ on $\tau_xM/\tau_xM_i$ is a rotation and hence 
there exists $t\in T$ 
such that $dg=dt$, i.e. $d(gt^{-1})$ is the identity on $\tau_xM$.  
On the other hand, since $gt^{-1}$ is contained in $G$ and $G$ is compact, 
the fixed point set of $gt^{-1}$ is a closed submanifold of $M$.  
The connected component of this submanifold containing $x$ is of codimension 0 because 
$d(gt^{-1})$ is the identity on $\tau_xM$.  
Since $M$ is connected, the connected component must agree with $M$ and 
this means $g=t$, proving the lemma. 
\end{proof}

\begin{coro} \label{G/G0}
Let $G$ be a compact Lie subgroup of $\Sympl(M,\omega)$ containing the torus $T$ and let $\G0$ be 
the identity component of $G$.  Then 
\[
G/\G0\cong \rhoP(\NGT)/\rhoP(\NG0T)\subset \Aut(P)/\rhoP(\NG0T)
\]
where $\rhoP$ is the map in \eqref{Phi} and $\NG0T$ is the normalizer of $T$ in $\G0$.  
\end{coro}

\begin{proof} 
Since $T$ is a maximal torus of $G$ and maximal tori in $G$ are conjugate to each other because 
$G$ is compact, the inclusion $\NGT\to G$ induces an isomorphism 
\begin{equation} \label{NGT}
\NGT/\NG0T\cong G/\G0.
\end{equation}
This fact together with Lemm~\ref{lemm2} implies the corollary.
\end{proof}

We shall construct a cross section of the homomorphism $\rhoP$ in \eqref{Phi} 
when $G=\Sympl(M,\omega)$. 
We recall the description \eqref{P2} of $P$:
\begin{equation} \label{6P}
P=\{u\in H^2(BT;\R)\mid \langle u,v_i\rangle\ge a_i\quad(i=1,\dots,m)\}. 
\end{equation}
 Let $\rho\in \Aut(P)$.  Since $\rho^*(P)\equiv P$, we have 
\[
\rho^*(P)=P+u_0 \quad\text{for some $u_0\in H^2(BT;\R)$}.
\]
On the other hand, we have 
\begin{equation*}  
\rho^*(P)=\{ \rho^*(u)\in H^2(BT;\R)\mid \langle u,v_i\rangle \ge a_i\quad(i=1,\dots,m)\}
\end{equation*}
by definition and this can be rewritten as  
\begin{equation}  \label{rho*P}
\rho^*(P)=\{ u\in H^2(BT;\R)\mid \langle u,{\rho_*}^{-1}(v_i)\rangle \ge a_i\quad(i=1,\dots,m)\}.
\end{equation}
Since $\rho^*(P)\equiv P$, it follows from \eqref{6P} and 
\eqref{rho*P} that there exists a permutation $\sigma$ on $[m]$ such that 
\begin{equation} \label{rho*-1}
\text{${\rho_*}^{-1}(v_i)=v_{\sigma(i)}$ for any $i\in[m]$} 
\end{equation}
so that 
\begin{equation*}
\rho^*(P)=\{ u\in H^2(BT;\R)\mid \langle u,v_{\sigma(i)}\rangle \ge a_i\quad(i=1,\dots,m)\}.
\end{equation*}
Therefore 
\[
\begin{split}
 \langle u,v_{\sigma(i)}\rangle\ge a_{\sigma(i)} &\Longleftrightarrow u\in P \\
&\Longleftrightarrow u+u_0\in \rho^*(P)\\
&\Longleftrightarrow  \langle u+u_0,v_{\sigma(i)}\rangle \ge a_i\\
&\Longleftrightarrow \langle u,v_{\sigma(i)}\rangle \ge a_i-\langle u_0,v_{\sigma(i)}\rangle
\end{split}
\]
and this shows that 
\begin{equation} \label{asi}
a_{\sigma(i)}=a_i-\langle u_0,v_{\sigma(i)}\rangle \quad\text{for any $i$}.
\end{equation}

We make one more observation on the permutation $\sigma$. 

\begin{lemm} \label{barf}
There is a ring automorphism of $H^*(M)$ sending $\mu_i$ to $\mu_{\sigma(i)}$ for each $i$. 
\end{lemm}

\begin{proof}
Since $\sigma$ is induced from the automorphism $\rho$ of $P$, we note that 
$\cap_{i\in I}P_i=\emptyset$  if and only if $\cap_{j\in \sigma(I)}P_j=\emptyset$ for $I\subset [m]$. 
Furthermore we note that $\cap_{i\in I}P_i=\emptyset$ if and only if $\cap_{i\in I}M_i=\emptyset$.  
Therefore, Lemma~\ref{ring} ensures that 
sending $\tau_i$ to $\tau_{\sigma(i)}$ for each $i$ induces a ring automorphism of $H^*_T(M)$, 
which we denote by $f$.  

Applying $f$ to the both sides of \eqref{pi*}, we have   
\begin{equation} \label{fpi}
f(\pi^*(u))=\sum_{i=1}^m\langle u,v_i\rangle f(\tau_i)=\sum_{i=1}^m \langle u,v_i\rangle\tau_{\sigma(i)},
\end{equation}
while it follows from \eqref{pi*} applied to $\rho^*(u)$ instead of $u$ that we have 
\begin{equation} \label{pirho}
\begin{split}
\pi^*(\rho^*(u))&=\sum_{i=1}^m \langle \rho^*(u),v_i\rangle\tau_i 
=\sum_{i=1}^m \langle \rho^*(u),v_{\sigma(i)}\rangle\tau_{\sigma(i)}\\ 
&=\sum_{i=1}^m \langle u,\rho_*(v_{\sigma(i)})\rangle\tau_{\sigma(i)}
=\sum_{i=1}^m \langle u,v_{i}\rangle\tau_{\sigma(i)}
\end{split}
\end{equation}
where we used \eqref{rho*-1} at the last identity. 
Comparing \eqref{fpi} with \eqref{pirho}, we obtain the identity $f(\pi^*(u))=\pi^*(\rho^*(u))$ for any $u\in H^2(BT)$ 
and this shows that the ring automorphism $f$ of $H^*_T(M)$ preserves the subalgebra 
$\pi^*(H^*(BT))$.  Therefore $f$ induces a ring automorphism $\bar f$ of $H^*(M)$ by 
Proposition~\ref{DaJu}.  Since $f(\tau_i)=\tau_{\sigma(i)}$ and $\mu_i=\iota^*(\tau_i)$ by \eqref{mu}, 
we have $\bar f(\mu_i)=\mu_{\sigma(i)}$ which proves the lemma. 
\end{proof}

We now define the unitary transformation $F_\rho$ of $\C^m$ by 
\begin{equation} \label{Frho}
F_\rho(z_1,\dots,z_m):=(z_{\sigma(1)},\dots,z_{\sigma(m)}).
\end{equation} 
It preserves $\mathcal{Z}_P$ because 
\[
\begin{split}
z\in\mathcal{Z}_P &\Longleftrightarrow \sum(\frac{1}{2}|z_{\sigma(i)}|^2+a_{\sigma(i)})\mu_{\sigma(i)}=0 
\quad(\text{by \eqref{ZP2}})\\
&\Longleftrightarrow \sum(\frac{1}{2}|z_{\sigma(i)}|^2+a_{i}-\langle u_0,v_{\sigma(i)}\rangle)\mu_{\sigma(i)}=0
\quad(\text{by \eqref{asi}})\\
&\Longleftrightarrow \sum(\frac{1}{2}|z_{\sigma(i)}|^2+a_{i})\mu_{\sigma(i)}=0\quad(\text{by (2) in Proposition~\ref{DaJu}})\\
&\Longleftrightarrow \sum(\frac{1}{2}|z_{\sigma(i)}|^2+a_{i})\mu_{i}=0\quad(\text{by Lemma~\ref{barf}})\\
&\Longleftrightarrow F_\rho(z)\in\mathcal{Z}_P\quad(\text{by \eqref{ZP2} and \eqref{Frho}}).
\end{split}
\]

Let $\phi$ be the automorphism of $(S^1)^m$ defined by 
\begin{equation} \label{eqphi}
\phi(g_1,\dots,g_m):=(g_{\sigma(1)},\dots,g_{\sigma(m)}).
\end{equation}
Then the map $F_\rho$ is $\phi$-equivariant.

\begin{lemm} \label{phi}
Let $\V\colon (S^1)^m\to T$ be the homomorphism in \eqref{Vlambda}.  
Then $\mathcal V\circ \phi=\rho\circ\mathcal V$.  In particular $\phi$ preserves $\Ker \mathcal V$.
\end{lemm}

\begin{proof}
Noting that $\rho(\lambda_v(g))=\lambda_{\rho_*(v)}(g)$, we see 
from \eqref{eqphi}, \eqref{Vlambda} and \eqref{rho*-1} that 
\[
\begin{split}
\mathcal V\big(\phi & (g_1,\dots,g_m)\big)
=\mathcal V(g_{\sigma(1)},\dots,g_{\sigma(m)})
=\prod_{i=1}^m \lambda_{v_i}(g_{\sigma(i)})\\
&=\prod_{i=1}^m \lambda_{\rho_*(v_{\sigma(i)})}(g_{\sigma(i)})
=\rho\big(\prod_{i=1}^m \lambda_{v_{\sigma(i)}}(g_{\sigma(i)})\big)\\
&=\rho\big(\prod_{i=1}^m \lambda_{v_{i}}(g_{i})\big)=\rho\big(\mathcal V(g_1,\dots,g_m)\big).
\end{split}
\]
This proves the lemma. 
\end{proof}

Since $M=\mathcal{Z}_P/\Ker\V$ and $F_\rho$ is $\phi$-equivariant, 
$F_\rho$ induces a diffeomorphism $\bar F_\rho$ of $M$ by Lemma~\ref{phi}. 
By definition $F_\rho$ is a unitary transformation on $\C^m$, so 
$\bar{F}_\rho$ preserves the symplectic form $\omega$ and hence 
$\bar{F}_\rho\in \Sympl(M,\omega)$. 

Finally we need to prove the following. 

\begin{lemm} \label{barFrho}
$\bar{F}_\rho$ normalizes $T$ and $\rhoP(\bar{F}_\rho)=\rho$. 
\end{lemm}

\begin{proof}
We view an element $g=(g_1,\dots,g_m)$ of $(S^1)^m$ as a diffeomorphism of $\mathcal{Z}_P\subset \C^m$. 
Then $F_\rho\circ g\circ F_\rho^{-1}=\phi(g)$. This identity decends to an identity 
\begin{equation} \label{norma}
\bar{F}_\rho\circ \V(g)\circ \bar{F}_\rho^{-1}=\V(\phi(g)) \quad\text{in $\Sympl(M,\omega)$}.
\end{equation}
Since $T=\V((S^1)^m)$, the identity \eqref{norma} shows that $\bar{F}_\rho$ normalizes $T$. 

Let $t\in T$. Then 
\begin{equation} \label{tVg}
\text{$t=\V(g)=\prod_{i=1}^m\lambda_{v_i}(g_i)$\quad  for some $g\in (S^1)^m$}
\end{equation} 
where \eqref{Vlambda} is used for the latter identity.  
Using \eqref{tVg} together with the definition of $\rhoP$ (see also \eqref{defrhog}), 
\eqref{norma}, \eqref{eqphi} and \eqref{rho*-1}, we have   
\[
\begin{split}
(\rhoP(&\bar{F}_\rho))(t)=\bar{F}_\rho\circ t\circ \bar{F}_\rho^{-1}=\bar{F}_\rho\circ \V(g)\circ \bar{F}_\rho^{-1}=\V(\phi(g))\\
&=\V(g_{\sigma(1)},\dots,g_{\sigma(m)})=\prod_{i=1}^m\lambda_{v_i}(g_{\sigma(i)})=\prod_{i=1}^m\lambda_{\rho_*(v_{\sigma}(i))}(g_{\sigma(i)})\\
&=\rho(\prod_{i=1}^m\lambda_{v_{\sigma(i)}}(g_{\sigma(i)}))=\rho(\prod_{i=1}^m\lambda_{v_{i}}(g_{i}))=\rho(\V(g))=\rho(t)
\end{split}
\]
and this proves the latter statement in the lemma. 
\end{proof}

\section{Maximal compact Lie subgroup of $\Sympl(M,\omega)$} \label{sect:7}

The purpose of this section is to prove the following. 

\begin{theo} \label{Gmax}
If a compact Lie subgroup $G$ of $\Sympl(M,\omega)$ containing the torus $T$ satisfies 
the following two conditions:
\begin{enumerate}
\item $\RG=\RM$, and
\item the map $\rhoP$ in \eqref{Phi} is surjective,
\end{enumerate}
then $G$ is maximal among compact Lie subgroups of $\Sympl(M,\omega)$ containing the torus $T$. 
Moreover, there is a compact Lie subgroup $G_{\max}$ of $\Sympl(M,\omega)$ which satisfies the 
conditions (1) and (2) above.  
\end{theo}

\begin{proof}
What we prove for the former part of the theorem 
is that if a compact Lie subgroup $H$ of $\Sympl(M,\omega)$ contains the $G$ in the theorem, then $H=G$.  

Since $G$ is a subgroup of $H$, the root system $\RG$ of $G$ is a subsystem of the root system $\Delta(H)$ of $H$. 
On the other hand, since $H$ is a compact Lie subgroup of $\Sympl(M,\omega)$ containing $T$, 
$\Delta(H)$ is a subsystem of $\RM$.  Therefore, it follows from the condition (1) in the theorem 
that $\RG=\Delta(H)$ and this shows that $\G0=H^0$ where the superscript $0$ denotes the identity 
components as before.   

Since $G$ is a subgroup of $H$, $\NGT$ is a subgroup of $N_H(T)$.  It follows from Lemma~\ref{lemm2} 
and the condition (2) in the theorem that $\NGT=N_H(T)$ and this together with the isomorphism 
\eqref{NGT} for $G$ and $H$ implies $G=H$ because $\G0=H^0$.  

The proof of the latter part of the theorem is as follows. 
Let $\tilde G_{\max}$ be the subgroup   of $\U(m)$ generated by 
$\tilde G$ in Proposition~\ref{mconn} and $F_\rho$'s in \eqref{Frho} (regarded as elements of $\U(m)$) 
for $\rho\in\Aut(P)$.  The identity component of $\tilde G_{\max}$ is $\tilde G$. 
The action of $\tilde G_{\max}$ on $\C^m$ leaves $\mathcal{Z}_P$ invariant and 
induces an effective action of $G_{\max}:=\tilde G_{\max}/\Ker\V$  
on $M$ preserving $\omega$.  The group $G_{\max}$ contains the torus $T$ and 
satisfies the two conditions in the theorem by Proposition~\ref{mconn} and Lemma~\ref{barFrho}. 
\end{proof}

\bigskip

\noindent
{\bf Acknowledgment.}  
I would like to thank Nigel Ray for his invitation to University of Manchester in summer of 2008 
and stimulating discussions.  This work is an outcome of a project with him.  
I also would like to thank Michael Wiemeler for explaining his work \cite{wiem08}  
which motivated the introduction of the root system $R(P)$. Finally I would like to thank 
Megumi Harada and Shintaro Kuroki for their helpful comments on an earlier version of the paper.

\end{document}